\documentclass[11pt,reqno]{amsart}

\newtheorem{proposition}{Proposition}[section]
\newtheorem{lemma}[proposition]{Lemma}
\newtheorem{corollary}[proposition]{Corollary}
\newtheorem{theorem}[proposition]{Theorem}

\theoremstyle{definition}
\newtheorem{definition}[proposition]{Definition}
\newtheorem{example}[proposition]{Example}

\newtheorem{remark}[proposition]{Remark}
\newtheorem{remarks}[proposition]{Remarks}

\newcommand{\thlabel}[1]{\label{th:#1}}
\newcommand{\thref}[1]{Theorem~\ref{th:#1}}
\newcommand{\selabel}[1]{\label{se:#1}}

\newcommand{\lelabel}[1]{\label{le:#1}}
\newcommand{\leref}[1]{Lemma~\ref{le:#1}}
\newcommand{\prlabel}[1]{\label{pr:#1}}
\newcommand{\prref}[1]{Proposition~\ref{pr:#1}}
\newcommand{\colabel}[1]{\label{co:#1}}
\newcommand{\coref}[1]{Corollary~\ref{co:#1}}
\newcommand{\relabel}[1]{\label{re:#1}}

\newcommand{\exlabel}[1]{\label{ex:#1}}
\newcommand{\exref}[1]{Example~\ref{ex:#1}}
\newcommand{\delabel}[1]{\label{de:#1}}
\newcommand{\deref}[1]{Definition~\ref{de:#1}}
\newcommand{\eqlabel}[1]{\label{eq:#1}}
\newcommand{\equref}[1]{(\ref{eq:#1})}

\newcommand{\End}{{\rm End}}

\def\RR{{\mathbb R}}

\def\CC{{\mathbb C}}

\newcommand{\Cc}{\mathcal{C}}

\def\*C{{}^*\hspace*{-1pt}{\Cc}}
\def\text#1{{\rm {\rm #1}}}

\input xy
\xyoption {all} \CompileMatrices

\usepackage{amssymb}
\usepackage{color,amssymb,graphicx,amscd,amsmath,dsfont,stmaryrd,xcolor,comment}
\usepackage[pagebackref,colorlinks,urlcolor=blue,linkcolor=blue,citecolor=blue]{hyperref}

\begin{document}

\title[On the structure and classification of Bernstein algebras]
{On the structure and classification of Bernstein algebras}

\author{G. Militaru}
\address{Faculty of Mathematics and Computer Science, University of Bucharest, Str.
Academiei 14, RO-010014 Bucharest 1, Romania and Simion Stoilow Institute of Mathematics of the Romanian Academy, P.O. Box 1-764, 014700 Bucharest, Romania}
\email{gigel.militaru@fmi.unibuc.ro and gigel.militaru@gmail.com}

\thanks{This work was supported by a grant of the Ministry of Research,
Innovation and Digitization, CNCS/CCCDI--UEFISCDI, project number
PN-III-P4-ID-PCE-2020-0458, within PNCDI III}


\begin{abstract} Linear algebra tools are used to give a new approach to the open problem of the classification of Bernstein algebras. We prove that any Bernstein
algebra $(A, \omega)$ is isomorphic to a semidirect product $N \ltimes_{(\cdot, \, \Omega)} \, k$ associated to a commutative algebra $(N, \cdot)$
such that $(x^2)^2 = 0$, for all $x\in N$ and an idempotent endomorphism $\Omega = \Omega^2 \in {\rm End}_k (N)$ of $N$ satisfying two compatibility conditions. The set of types of $(1 + |I|)$-dimensional Bernstein algebras is parametrized by an explicitely constructed classified object. The automorphisms group of any Bernstein algebra is described as a subgroup of the canonical semidirect product of groups $(N, +) \ltimes {\rm GL}_k (N)$.
\end{abstract}

\maketitle

\section*{Introduction}
One of the pioneers of using mathematics to model genetics was Bernstein \cite{bern1, bern2, bern3} who formulated what is currently known as the \emph{Bernstein problem}, which consists in classifying all possible situations of the population genetics satisfying the \emph{Stationarity Principle}, i.e. conditions required to ensure that a population attains equilibrium after one generation (\cite{Lj0}, \cite[Section 4]{reed}, \cite[Chapter 9]{WB}). Later on, Etherington \cite{Eth} and Schafer \cite{Sch} pointed out the important role of non-associative algebras in modelling genetics (for an up-dated survey see \cite{reed}) introducing and studying a number of non-associative algebras
that arise from genetics such as baric, (special) train, gametic, zygotic or genetic algebras. Based on these ideas, Lyubich \cite{Lj} and Holgate \cite{Hol} restarted the study of the Bernstein problem and introduced the concept of \emph{Bernstein algebra} as a commutative algebra $A$ over a field $k$ of characteristic $\neq 2$ such that there exists a non-zero morphism of algebras $\omega: A \to k$ such that $(a^2)^2 = \omega (a)^2 \, a^2$, for all $a \in A$. This was a turning point which generated an explosion of interest in the study of this new class of algebras. With this concept in hand, Bernstein's problem can be generalized as:

\emph{For a given positive integer $n$, describle and classify, up to an isomorphism, all Bernstein algebras of dimension $n$.}

Bernstein algebras of dimension $\leq 3$ over the complex numbers $\CC$ have been classified by Holgate \cite[Section 4]{Hol} (see also \cite[Theorem 9.20]{WB}), and those of dimension $4$, also over the field $\CC$, by Lyubich \cite{Lj01} and over an arbitrary field of characteristic $\neq 2$ by Cort\'{e}s \cite{corte}. The Bernstein algebras of dimension $5$, over an algebraically closed field of characteristic $\neq 2$, were classified in \cite{cortem} only for two classes of algebras: (a) the reduced ones or (b) Bernstein-Jordan algebras. To the best of our knowledge these are also the only complete classifications of Bernstein algebras, with one notable exception, namely the class of so-called \emph{simplicial stochastic} Bernstein algebras, which are completely
classified by Guti\'{e}rrez-Fern\'{a}ndez \cite{Fer1}: these are Bernstein algebras, defined over the field $\RR$, that admit a stochastic basis (for details see \cite{Fer1, Fer2}). However, even over the field $\RR$, there exist Bernstein algebras that do not admit a stochastic basis \cite{gonz}, thus the classification of all Bernstein algebras over an arbitrary field remains an open and tempting problem.

Most structure or classification results obtained so far for Bernstein algebras have intensively used their Peirce
decomposition: any Bernstein algebra $(A, \omega)$ has a non-zero idempotent $e$ and thus $A = k e \oplus U_e \oplus Z_e$,
where $U_e := \{x \in A \, |\,  2 e x = x \}$ and $Z_e := \{x \in A \, |\, e x = 0 \}$. This paper deals with Bernstein algebras following a different view-point
inspired by \cite[Section 2]{gonz2}: we shall not use the Peirce decomposition at all and instead we shall use only linear algebra tools and new constructions similar to those used for other classes of algebras \cite{am-2019, am-2016, am-2015, am-2013}. The first main result is \thref{structura} that gives the structure of Bernstein algebras (in particular for normal ones): any Bernstein algebra $(A, \omega)$ of dimension $1 + |I|$ is isomorphic to the algebra having $\{f, \, e_i \, | \, i\in I\}$ as a basis and the multiplication $\circ$ given for any $i$, $j\in I$ by:
\begin{equation}\eqlabel{ooo}
f^2 := f, \qquad e_i \circ e_j := e_i \cdot e_j, \qquad e_i \circ f = f \circ e_i := \frac{1}{2} \, \Omega (e_i)
\end{equation}
where $\cdot$ is a comutative algebra stucture on $N : = \oplus_{i\in I} \, k \, e_i$ satisfying the compatibiliy $(x^2)^2 = 0$, for all $x\in N$ (such an algebra we called it
a \emph{$4$-algebra}) and $\Omega = \Omega^2 \in {\End}_k (N)$ is an idempotent endomorphism of $N$ such that
\begin{equation*} \eqlabel{berdat0}
x^2 \cdot \Omega (x) = 0, \qquad \Omega(x)^2 + \Omega(x^2) = x^2
\end{equation*}
for all $x\in N$. Such an idempotent $\Omega$ we called it a \emph{Bernstein operator} on the $4$-algebra $(N, \cdot)$.
The weight of the algebra defined by \equref{ooo} is given by $\omega (f) := 1$ and $\omega (e_i) := 0$, for all $i\in I$. \thref{calssific} describes the set of all morphisms between two such Berstain algebras: in particular,
we give necessary and sufficient conditions for two Bernstein algebras to be isomorphic. Based on this, the automorphism group of any Bernstein algebra is described explicitly as a subgroup of the canonical semidirect product of groups $(N, +) \ltimes {\rm GL}_k (N)$ in \coref{autodesc}. The second main result is \thref{clasifgen}: the set of types of isomorphism of all Bernstein algebras of dimension $1 + |I|$ is explicitly described and parameterized. The results of this paper provided a new and efficient method of classification of finite dimensional Bernstein algebras, different from the one existing in the literature \cite{corte, Hol, Lj01, WB}: several illustrative examples are given and a list of open problems arising from our approach is proposed. For instance, applying our approach and the classical Jordan theory we prove that there exist precisely $n+1$ types of isomorphism of Bernstein algebras of dimension $n+1$ whose bar-ideal is an abelian algebra and their automorphisms groups are explicitly described.

\section{Preliminaries}\selabel{preli}
During this paper $k$ will be an infinite field of characteristic $\neq 2$ and all vector spaces, (bi)linear maps are taken over $k$.
For a family of sets $(A_i)_{i\in I}$ we shall denote by $\amalg_{i\in I} \, A_i$ their
coproduct, i.e. $\amalg_{i\in I} \, A_i$ is the disjoint union of all sets $A_i$. If $V$ and $W$ are two vector spaces, ${\rm Hom}_k (V, \, W)$ denote the  vector space of all linear maps $V \to W$; ${\rm End}_k \, (V)$ is the usual associative and unital endomorphisms algebra of $V$ and ${\rm GL}_k (V)$ is the automorphisms group of $V$. We denote by $(V, +) \ltimes {\rm GL}_k (V) := V \times {\rm GL}_k (V)$ the canonical semidirect product of groups having the multiplication given for any $(w, g)$ and $(v, f) \in V \times {\rm GL}_k (V)$ by:
\begin{equation}\eqlabel{semidirect}
(w, \, g) \bullet (v, \,  f) := (w + g(v), \, g\circ f).
\end{equation}

By an algebra $A = (A, \cdot)$, we mean a vector space $A$ with a bilinear map, called multiplication, $ \cdot : A \times A \to A$. The concepts of subalgebras, ideals, morphisms of algebras are defined in the standard way: for basic facts on non-associative algebras we refere to \cite{Zh2}. An algebra $A$ is called commutative if $a\cdot b = b\cdot a$, for all $a$, $b\in A$ and $A$ is called \emph{abelian} or \emph{zero algebra} if $a \cdot b = 0$, for all $a$, $b\in A$. A \emph{baric algebra} is a pair $(A, \omega)$ consisting of a commutative algebra $A$ with a non-zero morphism of algebras $\omega: A \to k$, called \emph{weight}. A baric algebra $(A, \omega)$ is called a \emph{Bernstein algebra} if:
\begin{equation}\eqlabel{Bern}
(a^2)^2 = \omega (a)^2 \, a^2
\end{equation}
for all $a\in A$. If $ A =(A, \omega)$ is a Bernstein algebra, then the weight $\omega: A \to k$ is unique \cite[Lemma 9.3]{WB} and ${\rm Ker} (\omega)$ is called the bar-ideal of $A$. A baric algebra $(A, \omega)$ is called a \emph{normal Bernstein algebra} \cite{Lj1} if:
\begin{equation} \eqlabel{Lj}
a^2 \cdot b = \omega (a) \, a \cdot b
\end{equation}
for all $a$, $b\in A$. Any normal Bernstein algebra is a Bernstein algebra and moreover, is a special train algebra, in particular a genetic algebra \cite[Theorem 9.16]{WB}. The relation between Bernstein algebras and Jordan algebras was studied in \cite{Lj0, WB1}: in particular, \cite[Theorem 7]{WB1} proves that any normal Bernstein algebra is a Jordan algebra. A morphism between two Bernstein algebras $(A, \omega_A)$ and $(B, \omega_B)$ is just a morphism of baric algebras, that is an algebra map $\psi: A \to B$ such that $\omega_B \circ \psi = \omega_A$. By ${\rm Aut}_{\rm Ber} (A)$ we denote the automorphism group of a Bernstein algebra $A$. For other basic results on Bernstein algebras we refer  to \cite{PZ, WB} and their references.

\section{The structure and classification of Bernstein algebras}\selabel{structure}

The following concepts will play the key role in this paper.

\begin{definition} \delabel{fouralg}
A \emph{$4$-algebra} is a commutative algebra $N = (N, \cdot)$ such that $(x^2)^2 = 0$, for all $x\in N$. A \emph{Bernstein operator} on a $4$-algebra
$(N, \cdot)$ is an idempotent $\Omega = \Omega^2 \in {\rm End}_k (N)$ satisfying the following compatibilities for any $x\in N$:
\begin{equation} \eqlabel{berdat}
x^2 \cdot \Omega (x) = 0, \qquad \Omega(x)^2 + \Omega(x^2) = x^2.
\end{equation}
The set of all Bernstein operators on $(N, \cdot)$ will be denoted by ${\mathcal BO} \, (N, \cdot)$. A \emph{normal Bernstein operator} on a $4$-algebra $(N, \cdot)$ is an idempotent $\Omega = \Omega^2 \in {\rm End}_k (N)$ such that for any $x$, $y\in N$:
\begin{equation} \eqlabel{norberdat}
\Omega (x^2) = 0, \qquad \Omega(x) \cdot y = x \cdot y.
\end{equation}
The set of all normal Bernstein operators on $(N, \cdot)$ will be denoted by ${\mathcal NBO} \, (N, \cdot)$.
\end{definition}

\begin{remarks} \relabel{4nor}
(1) The variety of $4$-algebras were introduced and studied in \cite{guzzo} where it was proved that any $4$-algebra of dimension $\leq 7$ is solvable and it
was conjectured that any finite dimensional $4$-algebra is solvable. We also mention that $4$-algebras are a special case of \emph{admissible cubic algebras} introduced and studied in \cite{el} in relation with Jordan algebras: these are commutative algebras $A$ satisfying the compatibility $(a^2)^2 = \Upsilon (a) \, a$, for all $a \in A$, where $\Upsilon: A \to k$ is a cubic form. Algebras satisfying similar conditions have often been studied: for instance, associative algebras satisfying the identity $x^5 = 0$ were studied in \cite{She} related to Kuzmin's conjecture on the index of nilpotency. If $(A, \omega)$ is a Bernstein algebra, then its bar-ideal ${\rm Ker} (\omega)$ is a $4$-algebra.

(2) We can easily prove that any normal Bernstein operator is a Bernstein operator, i.e. ${\mathcal NBO} \, (N, \cdot) \subseteq {\mathcal BO} \, (N, \cdot)$. The trivial idempotents $0$ and ${\rm Id}_N$ are Bernstein operators on a $4$-algebra $(N, \cdot)$ if and only if $(N, \cdot)$ is the abelian algebra. Indeed, if $0$ or ${\rm Id}_N$ are Bernstein operators on $N$ we obtain from the second compatibility condition of \equref{berdat} that $x^2 = 0$, for all $x\in N$: linearizing this condition we obtain that $x \cdot y = 0$, for all $x$, $y\in N$. Several examples of Bernstein operators are given at the end of the paper.

(3) Linearizing the second compatibility of \equref{berdat} we obtain that an idempotent $\Omega \in {\rm End}_k (N)$ is a Bernstein operator on a $4$-algebra $N$ if and only if for any $x$, $y\in N$:
\begin{equation*}
x^2 \cdot \Omega (x) = 0, \qquad \Omega(x) \cdot \Omega(y) + \Omega(x \cdot y) = x\cdot y.
\end{equation*}

\end{remarks}

The terminology of \deref{fouralg} is motivated by the following:

\begin{proposition} \prlabel{prop1}
Let $N$ be a vector space, $\cdot : N \times N \to N$ a bilinear map and $\Omega \in {\rm End}_k (N)$ a linear endomorphism of $N$. We denote
$N \ltimes_{(\cdot, \, \Omega)} k := N \times k$, as a vector space, with the multiplication $\circ$ given for any $x$, $y \in N$ and $\alpha$, $\beta \in k$ by:
\begin{equation} \eqlabel{inmultirea}
(x, \, \alpha) \circ (y, \, \beta) := (x \cdot y + \frac{1}{2} \alpha \, \Omega (y) + \frac{1}{2} \beta \, \Omega (x), \,\, \alpha\beta).
\end{equation}
Then $N \ltimes_{(\cdot, \, \Omega)} k$ is a (normal) Bernstein algebra if and only if $(N, \cdot)$ is a $4$-algebra and $\Omega$ is a
(normal) Bernstein operator on $(N, \cdot)$. The weight of $N \ltimes_{(\cdot, \, \Omega)} k$ is the canonical projection $\pi_2 : N \ltimes_{(\cdot, \, \Omega)} k \to k$,
$\pi_2 (x, \, \alpha) = \alpha$, for all $x\in N$, $\alpha\in k$.

The Bernstein algebra $N \ltimes_{(\cdot, \, \Omega)} k$ will be called the \emph{semidirect product} of $N$ and $k$ associated to $(\cdot, \, \Omega)$.
\end{proposition}

\begin{proof} First we observe that the multiplication on $N \times k$ given by \equref{inmultirea} is commutative if and only if $\cdot: N\times N \to N$ is a commutative algebra. Moreover, the canonical projection $\pi_2 : N \ltimes_{(\cdot, \, \Omega)} k \to k$, $\pi_2 (x, \, \alpha) = \alpha$ is an algebra map. Assume that $\cdot$ is commutative. Using the uniqueness of the weight of a Bernstein algebra \cite[Lemma 9.3]{WB} we obtain that $(N \ltimes_{(\cdot, \, \Omega)} k , \, \circ)$ is a Bernstein algebra if and only if
\begin{equation} \eqlabel{inmber}
\bigl((x, \, \alpha)^2\bigl)^2 = \pi_2 (x, \, \alpha)^2 (x, \, \alpha)^2
\end{equation}
for all $x \in N$ and $\alpha \in k$. Taking into account that $(x, \, \alpha)^2 = (x^2 + \alpha \, \Omega (x), \,\, \alpha^2)$ we obtain that \equref{inmber} holds if and only if
\begin{eqnarray*}
(x^2)^2 + 2 \, \alpha x^2 \, \cdot \Omega(x) + \alpha^2 \, \Omega(x)^2 + \alpha^2 \, \Omega(x^2) + \alpha^3 \, \Omega^2(x) = \alpha^2 \,  x^2 + \alpha^3 \, \Omega(x)
\end{eqnarray*}
for any $x\in N$ and $\alpha \in k$. Since $k$ is a field of characteristic $\neq 2$, the above identity holds for any $\alpha \in k$ if and only if
$$(x^2)^2 = 0, \qquad x^2 \cdot \Omega (x) = 0, \qquad \Omega(x)^2 + \Omega(x^2) = x^2, \qquad \Omega^2 = \Omega$$
for all $x\in N$. Thus, we have proved that $N \ltimes_{(\cdot, \, \Omega)} k$ is a Bernstein algebra if and only if $(N, \cdot)$ is a $4$-algebra and
$\Omega$ is a Bernstein operator on $(N, \cdot)$.

The case of normal Berstain algebra is proven analogously. Indeed, the compatibility condition:
\begin{equation} \eqlabel{Ljb}
(x, \, \alpha)^2 \cdot (y, \, \beta) = \pi_2 (x, \, \alpha) \,\, (x, \, \alpha) \cdot (y, \, \beta)
\end{equation}
holds for all $x$, $y\in N$ and $\alpha$, $\beta \in k$ if and only if
\begin{eqnarray*}
x^2 \cdot y + \alpha \, \Omega(x) \cdot y + \frac{1}{2} \beta \, \Omega(x^2) + \frac{1}{2} \alpha \beta \, \Omega^2 (x) = \alpha \, x\cdot y + \frac{1}{2}\alpha \beta \, \Omega(x)
\end{eqnarray*}
This identity holds for any $\alpha$, $\beta \in k$ if and only if
\begin{eqnarray*}
x^2 \cdot y = 0, \qquad \Omega(x^2) = 0, \qquad \Omega(x) \cdot y = x\cdot y, \qquad \Omega^2 = \Omega
\end{eqnarray*}
for all $x$, $y\in N$. Now, the first compatibility condition follows from the second and the third since $x^2 \cdot y = \Omega(x^2) \cdot y = 0$. Thus,
$N \ltimes_{(\cdot, \, \Omega)} k$ is a normal Bernstein algebra if and only if $(N, \cdot)$ is a $4$-algebra and
$\Omega$ is a normal Bernstein operator on $(N, \cdot)$ and the proof is finished.
\end{proof}

Let $\Omega \in {\mathcal BO}\, (N, \cdot)$ be a Bernstein operator on a $4$-algebra $(N, \cdot)$ having $\{e_i \, | \, i\in I\}$ as a basis. In the vector space $N\times \, k$ we identify $e_i = (e_i, \, 0)$ and we denote $f := (0, 1)$. Then the semidirect product $(N \ltimes_{(\cdot, \, \Omega)} \, k, \, \circ)$ is the Bernstein algebra
having $\{f, \, e_i \, | \, i\in I\}$ as a basis and the multiplication $\circ$ is given for any $i$, $j\in I$ by:
\begin{equation}\eqlabel{detaliat}
f^2 := f, \qquad e_i \circ e_j := e_i \cdot e_j, \qquad e_i \circ f = f \circ e_i := \frac{1}{2} \, \Omega (e_i).
\end{equation}
The weight of $(N \ltimes_{(\cdot, \, \Omega)} \, k, \, \circ)$ is given by $\omega (f) := 1$ and $\omega (e_i) := 0$, for all $i\in I$. In \thref{structura} we will prove that any  Bernstein algebra is isomorphic to such an algebra.

\begin{example} \exlabel{exberdatum}
Let $N$ be a vector space and assume that the bilinear map $\cdot$ is trivial, i.e. $x\cdot y = 0$, for all $x$, $y\in N$.
Then $\Omega$ is a (normal) Bernstein operator on $N$ if and only if $\Omega^2 = \Omega$, i.e. ${\mathcal BO} \, (N, \cdot) = {\mathcal NBO} \, (N, \cdot)
= {\rm Idem} \bigl(\rm End_k (N)\bigl)$. The associated Bernstein algebra $N \ltimes_{(\cdot = 0, \, \Omega)} \, k$
will be denoted by $\mathcal{B}_0 \, (N, \, \Omega)$ and was called in \cite{WB2} \emph{trivial Bernstein algebra} associated to an arbitrary idempotent
endomorphism $\Omega =\Omega^2$ of $N$. In particular, the Bernstein algebras associated to the trivial idempotents $0$ and ${\rm Id}_N$, namely
$\mathcal{B}_0 \, (N, \, \Omega := 0)$  (resp. $\mathcal{B}_0 \, (N, \, \Omega := {\rm Id}_N)$) were called the \emph{constant Bernstein algebra}
(resp. the \emph{unit Bernstein algebra}) \cite{Lj1}.
\end{example}

The proof of the following lemma is straightforward, but we need it in the proof of \thref{structura}:

\begin{lemma} \lelabel{transf}
Let $(A, \omega_A)$ be a (normal) Bernstein algebra and $(B, \omega_B)$ a baric algebra. If $\psi: B \to A$ is an isomorphism of baric algebras, then $(B, \omega_B)$ is a
(normal) Bernstein algebra.
\end{lemma}

\begin{theorem} \thlabel{structura}
Let $A = (A, \, \cdot_A, \, \omega)$ be a (normal) Bernstein algebra and $N := {\rm Ker} (\omega)$. Then there exists a $4$-algebra structure on $N = (N, \cdot)$, a (normal) Bernstein operator $\Omega$ on $(N, \cdot)$ and an isomorphism of Bernstein algebras $\psi: N \ltimes_{(\cdot, \, \Omega)} \, k \to A$ that stabilize $N$, i.e. $\psi (x, \, 0) = x$, for all $x \in N$.
\end{theorem}

\begin{proof}
Since $\omega \neq 0$ we can find $x\in A$ such that $\omega (x) = 1$. Then, using \equref{Bern} we obtain that $e : = x^2$ is an idempotent of $A$ and
$\omega (e) = 1$. Fix such an element $e\in A$. Then, we can easily see that the linear map $\psi : N
\times k \to A$, $\psi (x, \, \alpha) := x + \alpha \, e $ is an isomorphism of
vector spaces with the inverse $\psi^{-1} (y) = \bigl(y - \omega (y)\, e, \,\, \omega(y) \bigl)$, for all $y\in A$. We define the following two bilinear maps:
\begin{eqnarray*}
\cdot \, &:& N \times N \to N,
\,\,\,\,\, x \cdot y := x \cdot_A y \eqlabel{act2}\\
\Omega \, &:& N \to N, \,\,\,\,\,\,\,\,\, \Omega(x) := 2 \, e \cdot_A x \eqlabel{coc1}
\end{eqnarray*}
for all $x$, $y \in N$. They are well defined (i.e. $x \cdot y \in N$ and $\Omega(x) \in N$) since $\omega$ is an algebra map. As $A$ is a Bernstein algebra (resp. a normal Bernstein algebra) we will prove that $\cdot$ is a $4$-algebra structure on $N$, $\Omega$ is a Bernstein operator (resp. a normal Bernstein operator) on $(N, \cdot)$ and
$\psi: N \ltimes_{(\cdot, \, \Omega)} \, k \to A$ is an isomorphism of Bernstein algebras. Instead of proving the compatibility conditions \equref{berdat} (resp. \equref{norberdat}), we use the following trick combined with \leref{transf} and \prref{prop1}: $\psi : N \times k \to A$, $\psi (x, \, \alpha) := x + \alpha \, e $ is an isomorphism of vector spaces. Thus, there exists a unique
algebra structure on $N \times k$ such that $\psi$ is an isomorphism of non-associative algebras. This unique multiplication $\ast$ on $N \times k$ is given for any
$x$, $y \in N$ and $\alpha$, $\beta \in k$ by:
\begin{eqnarray*}
(x, \alpha) \ast (y, \beta) &=& \psi^{-1} \bigl(\psi(x, \alpha) \cdot_A
\psi(y, \beta)\bigl) = \psi^{-1} \bigl(  ( x + \alpha \, e ) \cdot_A (y + \beta \, e) \bigl) \\
&=& \psi^{-1} \bigl( x \cdot_A y + \alpha \, e \cdot_A y + \beta \, x \cdot_A e + \alpha\beta \, e^2 \bigl) \\
&=& \psi^{-1} \bigl( x \cdot y + \alpha \, e \cdot_A y + \beta \, e \cdot_A x + \alpha\beta \,  e \bigl) \\
&=& \bigl( x \cdot y + \alpha \, e \cdot_A y + \beta \, e \cdot_A x, \, \alpha\beta \bigl) \\
&=& \bigl( x \cdot y + \frac{1}{2} \alpha \, \Omega (y) + \frac{1}{2} \beta \, \Omega (x), \,\, \alpha\beta \bigl)
\end{eqnarray*}
This shows that the multiplication $\ast$ on $N\times k$ coincides with the one defined by
\equref{inmultirea}. We also observe that $\psi$ stabilize $N$ (that is $\psi (x, \, 0) = x$, for all $x \in N$) and $\psi$ is an isomorphism of
baric algebras since $(\omega \circ \psi) (x, \alpha) = \omega (x + \alpha \, e) = \alpha = \pi_2 (x, \alpha)$. Applying now \leref{transf}
and \prref{prop1} we obtain that $\psi: N \ltimes_{(\cdot, \, \Omega)} \, k \to A$ is an isomorphism of Bernstein algebras (resp. normal Bernstein algebras) and the proof is now finished.
\end{proof}

Based on \thref{structura} the classification of all (normal) Bernstein algebras is reduced to the classification of all semidirect products $N \ltimes_{(\cdot, \, \Omega)} \,k$.  The key step for this purpose is the folllowing:

\begin{theorem} \thlabel{calssific}
Let $(N, \cdot)$, $(M, \cdot')$ be two $4$-algebras, $\Omega \in {\mathcal BO}\, (N, \cdot)$ and $\Omega' \in {\mathcal BO}\, (M, \cdot')$ two Bernstein operators. Then there exists a bijection between the set of all morphisms of Bernstein algebras $\psi: N \ltimes_{(\cdot, \, \Omega)} \, k \to M \ltimes_{(\cdot', \, \Omega')} \, k$
and the set of all pairs $(v_0, \, f) \in M \times {\rm Hom}_k (N, \, M)$ satisfying the following compatibilities:
\begin{eqnarray}
f(x\cdot y) &=& f(x) \cdot' f(y) \eqlabel{Mor1}\\
f\bigl(\Omega (x)\bigl) - \Omega'\bigl(f(x)\bigl)  &=& 2 \, f(x) \cdot' v_0 \eqlabel{Mor2}\\
\Omega' (v_0) &=& v_0 - v_0 \cdot' v_0 \eqlabel{Mor3}
\end{eqnarray}
for all $x$, $y\in N$. Under the above bijection the Bernstein algebras map $\psi =
\psi_{(v_0, f)} : N \ltimes_{(\cdot, \Omega)} k \to M \ltimes_{(\cdot', \Omega')} k$ corresponding to $(v_0, \, f) \in M \times {\rm Hom}_k (N, \, M)$
is given for any $x\in N$ and $\alpha \in k$ by:
\begin{equation}\eqlabel{morfis}
\psi (x, \, \alpha) = \bigl (f(x) + \alpha\, v_0, \,\, \alpha\bigl)
\end{equation}
Moreover, $\psi = \psi_{(v_0, \, f)}: N \ltimes_{(\cdot, \, \Omega)} \, k \to M \ltimes_{(\cdot', \, \Omega')} \, k$ is an isomorphism of Bernstein algebras if and only if $f: (N, \cdot) \to (M, \cdot')$ is an isomorphism of $4$-algebras.
\end{theorem}

\begin{proof}
A linear map $\psi: N  \times \, k \to M \times \, k$ is uniquely determined by two
linear maps $\psi_1: N\times k \to M$ and $\psi_2: N\times k \to k $ such that $\psi (x, \alpha) = \bigl(\psi_1 (x, \alpha), \, \psi_2 (x, \alpha) \bigl)$,
for all $(x, \, \alpha) \in N\times k$. Such a map satisfies the baric condition $\pi_2 \circ \psi = \pi_2$ if and only if $\psi_2 (x, \alpha) = \alpha$,
for all $(x, \, \alpha) \in N\times k$. Let $f: N \to M$, $f (x) := \psi_1 (x, 0)$ and $v_0 := \psi_1 (0, 1) \in M$. Then, $\psi_1(x, \, \alpha) = f(x) + \alpha v_0$ and
thus any linear map $\psi : N  \times \, k \to M \times \, k$ is uniquely determined by a pair $(v_0, \, f) \in M \times {\rm Hom}_k (N, \, M)$ via the
formula \equref{morfis}. Now, by a straightforward computation we can prove that such a map
$\psi = \psi_{(v_0, f)} : N \ltimes_{(\cdot, \Omega)} \, k \to M \ltimes_{(\cdot', \Omega')} \, k$ is an algebra map if and only if the following condition holds for any $x$, $y\in N$ and $\alpha$, $\beta\in k$:
\begin{eqnarray*}
&& f(x \cdot y) + \frac{1}{2} \alpha \, f(\Omega(y)) + \frac{1}{2} \beta \, f(\Omega(x)) + \alpha\beta \, v_0 = f(x) \cdot' f(y) + \\
&& + \beta \, f(x) \cdot' v_0 + \alpha \, f(y) \cdot' v_0 + \alpha\beta \, v_0 \cdot' v_0 + \frac{1}{2} \alpha \, \Omega' (f(y)) + \frac{1}{2} \beta \, \Omega' (f(x)) + \alpha\beta \, \Omega'(v_0).
\end{eqnarray*}
The above compatibility holds for any $\alpha$, $\beta \in k$ if and only if the three compatibilities \equref{Mor1}-\equref{Mor3} hold. The last statement is elementary: we point out that if $f$ is bijective, the inverse of $\psi_{(v_0, f)}$ is given by $\psi_{(v_0, f)} ^{-1} (x, \, \alpha) = \bigl(f^{-1}(x) - \alpha f^{-1} (v_0), \, \alpha \bigl)$, for all $x\in M$ and $\alpha \in k$.
\end{proof}

\subsection*{Applications: the automophisms group and the classification of Bernstein algebras}

As a first application of the previous results we can explicitely describle the automophisms group of a given Bernstein algebra. Let $(N, \cdot)$ be a $4$-algebra,
$\Omega \in {\mathcal BO} \, (N, \cdot)$ a Bernstein operator on $(N, \cdot)$ and denote by ${\rm Aut}_{\rm Alg} (N, \cdot)$ the automophism group of the
commutative algebra $(N, \cdot)$. Let $\mathcal{A} \, (N, \cdot, \Omega)$ be the set of all pairs $(v, \, f) \in N \times {\rm Aut}_{\rm Alg} (N, \cdot)$ such that:
\begin{equation}\eqlabel{auto}
\Omega (v) = v - v^2, \quad f\bigl(\Omega(x)\bigl) - \Omega \bigl(f (x)\bigl) = 2\, f(x) \cdot v
\end{equation}
for all $x\in N$. Then, we can easily prove that $\mathcal{A} \, (N, \cdot, \Omega)$ has a group structure via the multiplication given by:
\begin{equation}\eqlabel{autoinmu}
(w, \, g) \bullet (v, \,  f) := (w + g(v), \, g\circ f)
\end{equation}
for all $(w, g)$ and $(v, f) \in \mathcal{A} (N, \cdot, \Omega)$ and $(\mathcal{A} (N, \cdot, \Omega), \, \bullet)$ is a subgroup in the canonical semidirect product of groups
$(N, +) \ltimes {\rm GL}_k (N)$ as defined by \equref{semidirect}. Applying \thref{calssific} for $M = N$ and $(\cdot', \, \Omega') = (\cdot, \, \Omega)$ we obtain:

\begin{corollary} \colabel{autodesc}
Let $(N, \cdot)$ be a $4$-algebra and $\Omega \in {\mathcal BO} \, (N, \cdot)$ a Bernstein operator on $(N, \cdot)$. Then the map defined for any
$(v, f) \in \mathcal{A} \, (N, \cdot, \Omega)$ and $(x, \alpha) \in N\times k$ by:
$$
\vartheta : (\mathcal{A} \, (N, \cdot, \Omega), \, \bullet) \to {\rm Aut}_{\rm Ber} \, (N \ltimes_{(\cdot, \Omega)} k), \qquad
\vartheta (v, f) (x, \, \alpha) := \bigl (f(x) + \alpha\, v_, \,\, \alpha\bigl)
$$
is an isomorphism of groups.
\end{corollary}

\begin{example} \exlabel{exauto}
Let $\mathcal{B}_0 \, (N, \, \Omega)$ be the trivial Bernstein algebra associated to an idempotent endomorphism $\Omega =\Omega^2$ of vector space $N$ as constructed in
\exref{exberdatum}. Then the automophism group ${\rm Aut}_{\rm Ber} \, \bigl(\mathcal{B}_0 \, (N, \, \Omega) \bigl)$ identifies with the subgroup of the
semidirect product $(N, +) \ltimes {\rm GL}_k (N)$ consisting of all pairs $(v, \, f) \in N \times {\rm GL}_k (N)$ satisfying the following two
compatibilities: $\Omega (v) = v$ and $f \circ \Omega = \Omega \circ f$. In particular, for the constant Bernstein algebra (resp. the unit Bernstein algebra) we obtain the following isomorphisms of groups:
$$
{\rm Aut}_{\rm Ber} \, \bigl(\mathcal{B}_0 \, (N, \, \Omega := 0 ) \bigl) \, \cong {\rm GL}_k (N), \quad
{\rm Aut}_{\rm Ber} \, \bigl(\mathcal{B}_0 \, (N, \, \Omega := {\rm Id}_N ) \bigl) \, \cong (N, +) \ltimes {\rm GL}_k (N)
$$
\end{example}

Now, we shall look at the classification problem for Bernstein algebras. For this purpose we introduce:

\begin{definition} \delabel{ecBO}
Let $(N, \cdot)$ be a $4$-algebra. Two Bernstein operators $\Omega$ and $\Omega' \in {\mathcal BO} \, (N, \cdot)$ on $(N, \cdot)$ are called
\emph{equivalent} and we denote this by $\Omega \approx \Omega'$ if there exists a pair $(v_0, \, f) \in N \times {\rm Aut}_{\rm Alg} (N, \cdot)$ such that
for any $x\in N$:
\begin{equation} \eqlabel{echividem}
\Omega' (x) = (f \circ \Omega \circ f^{-1}) (x) - 2\, x \cdot v_0, \qquad
\Omega' (v_0) = v_0 - v_0^2.
\end{equation}
\end{definition}

If follows from \thref{calssific} that $\Omega \approx \Omega'$ if and only if there exists an isomorphism of Bernstein algebras
$N \ltimes_{(\cdot, \, \Omega)} \, k \cong N \ltimes_{(\cdot, \, \Omega')} \, k$. Hence, $\approx$ is an equivalent relation of
the set ${\mathcal BO} \, (N, \cdot)$ and the quotient set is denoted by ${\mathcal BO} \, (N, \cdot)/ \approx$. The set of types
of isomorphisms of $4$-algebras of a given (possibly infinite) dimension $\mathfrak{c}$ will be denoted by ${\rm Types} \, (4, \mathfrak{c})$.
As a conclusion of the results of this paper we obtain the following classification result:

\begin{theorem} \thlabel{clasifgen}
Let $\mathfrak{c}$ be a given cardinal and ${\rm Bernstein} \, (1 + \mathfrak{c})$ the set of types of isomorphism of all Bernstein algebras of dimension $1 + \mathfrak{c}$.
Then the map:
\begin{equation} \eqlabel{clasifgen}
\xi : \amalg_{(N, \cdot)\in {\rm Types} \, (4, \mathfrak{c})} \, {\mathcal BO} \, (N, \cdot) /\approx \,\, \to {\rm Bernstein} \, (1 + \mathfrak{c}), \quad
\xi \bigl( \overline{\Omega} \bigl) \, := N \ltimes_{(\cdot, \, \Omega)} \, k
\end{equation}
is bijective, where the coproduct in the left hand side is taken over all $4$-algebras $(N, \cdot)\in {\rm Types} \, (4, \mathfrak{c})$ and $\overline{\Omega}$ is the equivalent class of $\Omega$ via the equivalent relation given by \equref{echividem}.
\end{theorem}

\begin{proof}
According to \thref{structura} we have that any Bernstein algebra of dimension $1 + \mathfrak{c}$ is isomorphic to
a semidirect product $N \ltimes_{(\cdot, \, \Omega)} \, k$, where $N = (N, \, \cdot)$ is a $\mathfrak{c}$-dimensional $4$-algebras and $\Omega$ is
a Bernstein operator on $N = (N, \, \cdot)$. It follows from \thref{calssific} that two non-isomorphic $4$-algebras
$(N, \cdot)$ and $(M, \cdot')$ of dimension $\mathfrak{c}$ give non-isomorphic Bernstein algebras
$N \ltimes_{(\cdot, \, \Omega)} \, k$ and $M \ltimes_{(\cdot', \, \Omega')} \, k$ - the coproduct from the left hand side of \equref{clasifgen} arise from this remark.
Hence, we can fix a $4$-algebra $(N, \cdot)$ of dimension $\mathfrak{c}$. Let $\Omega$ and $\Omega' \in {\mathcal BO} \, (N, \cdot)$ be two Bernstein operators on $(N, \cdot)$. Then, applying once again \thref{calssific}, we obtain that there exists
an isomorphism of Bernstein algebras $N \ltimes_{(\cdot, \, \Omega)} \, k \cong N \ltimes_{(\cdot, \, \Omega')} \, k$
if and only if $\Omega \approx \Omega'$ as introduced in \deref{ecBO} and this finishes the proof. We just mention that, for a given $4$-algebra $(N, \cdot)$,
the quotient set ${\mathcal BO} \, (N, \cdot) /\approx$ classifies all Bernstein algebras $(A, \omega)$ such that ${\rm Ker} (\omega) \cong (N, \cdot)$.
\end{proof}

\begin{remark} \relabel{clasLub}
The similiar classification result holds true for normal Bernstein algebras: the only change that has to be made in \thref{clasifgen} is the following:
the set ${\mathcal BO} \, (N, \cdot)$ has to be replaced by the set ${\mathcal NBO} \, (N, \cdot)$ of all normal Bernstein operators on $(N, \cdot)$.
\end{remark}

Computing the classifying object given by \equref{clasifgen} of \thref{clasifgen} is a very difficult problem. Among all components of the coproduct
$\amalg_{(N, \cdot)\in {\rm Types} \, (4, \mathfrak{c})} \, {\mathcal BO} \, (N, \cdot) /\approx$ the simplest is the one corresponding to  the
abelian algebra structure on $N$, $x \cdot y = 0$, for all $x$, $y\in N$. In this case, ${\mathcal BO} \, (N, \cdot := 0 ) /\approx$ is described below
and it classifies all Bernstein algebras $\mathcal{B}_0 \, (N, \, \Omega)$ as constructed in \exref{exberdatum}.

\begin{example} \exlabel{clasif1}
Let $N$ be a vector space, $\Omega = \Omega^2$ and $\Omega' = \Omega'^2 \in {\rm End}_k (N)$ two idempotents endomorphisms of $N$. Then, the Bernstein algebras
$\mathcal{B}_0 \, (N, \, \Omega)$ and $\mathcal{B}_0 \, (N, \, \Omega')$  are isomorphic if and only if $\Omega$ and $\Omega'$ are similar endomorphisms of $N$, i.e. there exists $f\in {\rm GL}_k (N)$ such that $\Omega' = f \circ \Omega \circ f^{-1}$. In particular, if ${\rm dim}_k (N) = n$, then
$$
{\mathcal BO} \, (N, \cdot := 0 ) /\approx \,\, = \{ \overline{0}, \,\, \overline{\Omega_i} \, | \, i= 1,\cdots, n  \}
$$
where, for any $i = 1, \cdots, n$, $\Omega_i$ is an endomorphism of $N$ whose canonical Jordan form is the $n\times n$-matrix $e_{11} + \cdots + e_{ii}$; we denoted by $e_{ij}$ the canonical basis of the matrix algebra ${\rm M}_n (k)$. In particular, we obtain that there exist precisely $n+1$ types of isomorphism of Bernstein algebras
of dimension $n+1$ whose bar-ideal is an abelian algebra. These are the Bernstein algebras $A_0$, $A_1$, ... $A_n$, where for any $t = 0, \cdots, n$, $A_t$ is the algebra having $\{f, \, e_i \, | \, i= 1, \cdots n\}$ as a basis and the multiplication $\circ$ is given for any $i = 1, \cdots, n$ by:
\begin{equation}\eqlabel{detaliat}
f^2 := f, \qquad e_i \circ f = f \circ e_i := \frac{1}{2} \, \Omega_t (e_i)
\end{equation}
where $\Omega_0 :=0$ and $\Omega_t := e_{11} + \cdots + e_{tt}$, for all $t=1, \cdots, n$ (undefined multiplications on the elements of a basis are all zero)\footnote{For a similar result, using a completely different aproach, see \cite[Section 2]{WB2}.}. Furthermore, the automorphisms group of the Bernstein algebras $A_t$ can be easily described by applying \coref{autodesc}: ${\rm Aut}_{\rm Ber} \, (A_t)$ identifies with the set of all pairs $(v, \, f) \in N \times {\rm GL}_k (N)$ such that $\Omega_t (v) = v$ and $f \circ \Omega_t = \Omega_t \circ f$, for all $t = 0, \cdots, n$.

Indeed, the Bernstein operators for the abelian algebra structure on $N$ are precisely the idempotent endomorphisms
of $N$, i.e. ${\mathcal BO} \, (N, \cdot := 0) = {\rm Idem} \, \bigl({\rm End}_k (N)\bigl)$. The equivalent relations \equref{echividem} come down to the following: $\Omega \approx \Omega'$ if and only if there exists a pair $(v_0, \, f) \in N \times {\rm GL}_k (N)$ such that $\Omega' (v_0) = v_0$ and $\Omega' = f \circ \Omega \circ f^{-1}$ (that is $\Omega$ and $\Omega'$ are similar endomorhisms of $N$). Hence the conclusion follows (for the converse we just take $v_0 := 0$). The last statement follows from the classical Jordan theory: the canonical Jordan form of an idempotent endomorphism has one of the forms described in the statement
since the minimal polynomial of such an idempotent endomorphism is a divisor of $X^2 - X$.
\end{example}

We give now an example of computing the classifying object ${\mathcal BO} \, (N, \cdot ) /\approx $ for a non-abelian $4$-algebra $(N, \cdot )$. Linear endomorphisms of a finite dimensional vector space $N$ will be written as matrices and undefined multiplications on the elements of a basis are all zero. The computations being very long, we will indicate only the essential steps, details being given upon request.

\begin{example} \exlabel{exnetri}
Let $(N, \cdot )$ be the $4$-algebra with a basis $\{e_1, \, e_2, \, e_3\}$ and the multiplication $e_1 \cdot e_2 = e_2 \cdot e_1 := e_1$. Then
$$
{\mathcal BO} \, (N, \cdot ) /\approx \,\, = \{ \overline{\Omega_1},  \,\,  \overline{\Omega_2}, \,\,  \overline{\Omega_2} \}
$$
where $\Omega_1$, $\Omega_2$, $\Omega_3$ are the following Bernstein operators on $N$:
\begin{equation*}\eqlabel{matrV}
\Omega_1 = \begin{pmatrix} 1 & 0 & 0  \\
0 & 0 & 0 \\
0 & 0 & 1
\end{pmatrix}, \quad
\Omega_2 = \begin{pmatrix} 1 & 0 & 0  \\
0 & 0 & 0 \\
0 & 0 & 0
\end{pmatrix}, \quad
\Omega_3 = \begin{pmatrix} 1 & 0 & 0  \\
0 & 0 & 0 \\
1 & 0 & 0
\end{pmatrix}
\end{equation*}
Hence, any $4$-dimensional Bernstein algebra $(A, \omega)$ such that ${\rm Ker} (\omega) \cong (N, \cdot)$ is isomorphic to one of the following three
algebras with the basis $\{f, \, e_1, \, e_2, \, e_3\}$ and the multiplications given by:
\begin{eqnarray*}
&N \ltimes_{(\cdot, \, \Omega_1)} \, k: & \, f^2 = f, \,\, e_1 \cdot e_2 = e_1, \,\,  e_1 \cdot f = \frac{1}{2}\, e_1, \,\,  e_3 \cdot f = \frac{1}{2}\, e_3,  \\
&N \ltimes_{(\cdot, \, \Omega_2)} \, k: & \, f^2 = f, \,\, e_1 \cdot e_2 = e_1, \,\,  e_1 \cdot f = \frac{1}{2}\, e_1, \\
&N \ltimes_{(\cdot, \, \Omega_3)} \, k: & \, f^2 = f, \,\, e_1 \cdot e_2 = e_1, \,\,  e_1 \cdot f = \frac{1}{2}\, e_1, \,\,  e_3 \cdot f = \frac{1}{2}\, e_1
\end{eqnarray*}

Indeed, by a long but straightforward computation we obtain that the set of all Bernstein operators on $(N, \cdot)$ is
${\mathcal BO} \, (N, \cdot ) = \{ \Omega_{\alpha, \beta, 1}, \,\, \Omega_{\gamma, \delta} \, |\, \alpha, \beta, \gamma, \delta \in k \}$ where:
\begin{equation*}
\Omega_{\alpha, \beta, 1} = \begin{pmatrix} 1 & 0 & 0  \\
\alpha & 0 & \beta \\
0 & 0 & 1
\end{pmatrix}, \quad
\Omega_{\gamma, \delta} = \begin{pmatrix} 1 & 0 & 0  \\
\gamma & 0 & 0 \\
\delta & 0 & 0
\end{pmatrix}
\end{equation*}
In the second step we have to compute the automorphism group ${\rm Aut}_{\rm Alg} (N, \cdot)$ of the $4$-algebra $(N, \cdot)$. Again, by a straightforward computation we obtain that ${\rm Aut}_{\rm Alg} (N, \cdot)$ is the subgroup of ${\rm GL} (3, k)$ consisting of all invertible $3\times 3$-matrices of the form:
\begin{equation} \eqlabel{auto3}
{\rm Aut}_{\rm Alg} (N, \cdot) \, \cong \,
\{ \begin{pmatrix} a & 0 & 0  \\
0 & 1 & b \\
0 & 0 & c
\end{pmatrix} \, |\, a, c \in k^*, \, b\in k    \}.
\end{equation}
Finally, in the last step it remains to classify all Bernstein operators via the equivalent relation \equref{echividem}. By a long but
routine computation we will obtain that $\Omega_{\alpha, \beta, 1} \approx \Omega_{0, 0, 1}$, for all $\alpha$, $\beta\in k$ and for any $\gamma$, $\delta \in k$ we have that $\Omega_{\gamma, \delta}$ is equivalent either with $\Omega_{0, 0}$ or with $\Omega_{0, 1}$. The proof is finished once we observe that any two Bernstein operators $\Omega_{0, 0, 1}$, $\Omega_{0, 0}$ and $\Omega_{0, 1}$ are not equivalent.

The automorphisms groups of the above three Bernstein algebras can be also described using \coref{autodesc}. A few computations give that
${\rm Aut}_{\rm Ber} \, (N \ltimes_{(\cdot, \, \Omega_1)} \, k ) \cong k \times k^* \times k^*$, the group with the multiplication given by:
$$
(v, \, a, \, c) \bullet (v', \, a', \, c') := (v + v' c, \, aa', cc')
$$
for all $v$, $v'\in k$ and $a$, $c$, $a'$, $c'\in k^*$. In a similar way, ${\rm Aut}_{\rm Ber} \, (N \ltimes_{(\cdot, \, \Omega_2)} \, k ) \cong
{\rm Aut}_{\rm Alg} \, (N, \cdot)$, the group of all invertible $3\times 3$-matrices of the form \equref{auto3}. Finally, we have an isomporhism of groups
${\rm Aut}_{\rm Ber} \, (N \ltimes_{(\cdot, \, \Omega_3)} \, k ) \cong k\times k^*$, the group with the multiplication given by:
$$
(v, \, a) \bullet (v', \, a') := (v + v' a, \, aa')
$$
for all $v$, $v'\in k$ and $a$, $a'\in k^*$.
\end{example}

\begin{remark} \relabel{insistareferendul}
The Bernstein algebras $N \ltimes_{(\cdot, \, \Omega_i)} \, k$, $i = 1, 2, 3$ as constructed in \exref{exnetri} are, up to an isomorphism, the only $4$-dimensional Bernstein algebras whose bar-ideal is the $4$-algebra $N$ as given there. On the other hand, as we mention in the introduction, the Bernstein algebras of dimension $4$ were classified
over an arbitrary field of characteristic $\neq 2$ in \cite[Theorem, pg.1431]{corte} using completely different methods. Our examples of Bernstein algebras fit the classification. Indeed, the first algebra $N \ltimes_{(\cdot, \, \Omega_1)} \, k$ is of type $(3, 1)$ in the Peirce decomposition \cite[Theorem, pg.1431]{corte} since we can easily prove that (with the notations of \cite{corte}) $V_f = k e_2$ and $U_f = k e_1 + k e_3$. Thus, $N \ltimes_{(\cdot, \, \Omega_1)} \, k$ belongs to the fourth family of Bernstein algebras as listed in \cite[Theorem, pg.1432, Table II]{corte}. In the same fashion we can prove that $N \ltimes_{(\cdot, \, \Omega_2)} \, k$ is a $4$-dimensional Bernstein algebra of type $(2, 2)$ since $V_f = k e_2 + k e_3$ and $U_f = k e_1$. Thus, $N \ltimes_{(\cdot, \, \Omega_2)} \, k$ appears as a special case of the third type of Bernstein algebras listed in \cite[Theorem, pg.1432, Table I]{corte}. The details are left to the reader. 
\end{remark}

\begin{remark} \relabel{clasmicdim}
Up to an isomorphism there exist two types of $2$-dimensional Bernstein algebras, namely the ones with the basis $\{e_1, \, f \}$ and the multiplication given by:
\begin{eqnarray*}
&A_0:& f^2 = f,\\
&A_1:& f^2 = f, \quad e_1 \circ f = f \circ e_1 = \frac{1}{2} e_1
\end{eqnarray*}
Indeed, the bar-ideal of a $2$-dimensional Bernstein algebra is an $1$-dimensional $4$-algebra $N = ke_1$ with the abelian multiplication, i.e. $e_1 \cdot e_1 = 0$. The conclusion follows from \exref{clasif1} and we recover the classification originally given by Bernstein \cite{bern2, bern3}.

Now, if we take a step further, any $3$-dimensional Bernstein algebra has a $2$-dimensional bar-ideal $N$. We have shown in \cite[Example 1.1]{mil22} that, up to an isomorphism,
there are exactly three $2$-dimensional $4$-algebras, namely the abelian one $N_0$ and the algebras with the basis $\{e_1, \, e_2\}$ and the multiplication $\cdot$ given by:
\begin{eqnarray*}
&N_1:& e_1^2 = e_2, \quad e_1 \cdot e_2 = e_2^2 = 0,\\
&N_2:& e_1 \cdot e_2 = e_2, \quad e_1^2 = e_2^2 = 0.
\end{eqnarray*}
It follows from \exref{clasif1} that there exist exactly three types of isomorphisms of $3$-dimensional Bernstein algebras whose bar-ideal is isomorphic
to $N_0$, namely the algebras $A_0$, $A_1$ and $A_2$ with the basis $\{e_1, \, e_2, \, f\}$ and the multiplication given by:
\begin{eqnarray*}
&A_0:& f^2 = f, \\
&A_1:& f^2 = f, \quad e_1 \circ f = f \circ e_1 = \frac{1}{2} e_1,\\
&A_2:& f^2 = f, \quad e_1 \circ f = f \circ e_1 = \frac{1}{2} e_1, \quad e_2 \circ f = f \circ e_2 = \frac{1}{2} e_2.
\end{eqnarray*}
The cases in which the bar-ideal of $3$-dimensional Bernstein algebras is $N_1$ and $N_2$ are treated similarly to \exref{exnetri} and are left to the reader. Thus, the classification of all $3$-dimensional Bernstein algebras, as proved by Holgate \cite[Section 4]{Hol} over $\CC$, is recovered.
\end{remark}

\subsection*{Conclusions and open problems} Let $n$ be a positive integer and $\{e_i \, | \, i = 1, \cdots, n\}$ the canonical basis of the vector
space $N:= k^n$. We have proved that any $(1+n)$-dimensional Bernstein algebra is isomorphic to
the algebra having $\{f, \, e_i \, | \, i = 1, \cdots, n\}$ as a basis and the multiplication $\circ$ given for any $i$, $j = 1, \cdots, n$ by:
\begin{equation}\eqlabel{detaliat0}
f^2 := f, \qquad e_i \circ e_j := e_i \cdot e_j, \qquad e_i \circ f = f \circ e_i := \frac{1}{2} \, \Omega (e_i)
\end{equation}
where $\cdot$ is a $4$-algebra stucture on $N : = k^n$ and $\Omega = \Omega^2 \in {\End}_k (k^n) \cong {\rm M}_n (k)$ is an idempotent endomorphism of $k^n$ satisfying the following compatibilities for any $x\in k^n$:
\begin{equation*} \eqlabel{berdat0}
x^2 \cdot \Omega (x) = 0, \qquad \Omega(x)^2 + \Omega(x^2) = x^2.
\end{equation*}
We also proved that the set of types of isomorphisms of all Bernstein algebras of dimension $1 + n$ is parameterized by the coproduct
$\amalg_{(k^n, \cdot)\in {\rm Types} \, (4, n)} \, {\mathcal BO} \, (k^n, \cdot) /\approx $. Thus, for a complete solution of the classification of all Bernstein algebras of a given dimension $n$, the following steps (which are subsequent open problems) must be followed:

\textbf{Question 1:} \emph{For a given positive integer $n$, describe and classify, up to an isomorphism, all $4$-algebras of dimension $n$. For a given $4$-algebra structure $\cdot$ on $k^n$:}

\emph{(i) Describe the set ${\mathcal BO} \, (k^n, \cdot)$ of all Bernstein operators on $(k^n, \cdot)$.}

\emph{(ii) Describe explicitly the automorphisms groups ${\rm Aut}_{\rm Alg} (k^n, \cdot)$ of $(k^n, \cdot)$. }

First steps on the classification of $4$-algebras were taken recently in \cite{mil22}. In this context, we also recall the conjecture formulated by Guzzo-Benh \cite{guzzo} that was proved in the affirmative sense up to dimension $7$: \emph{is any finite dimensional $4$-algebra solvable?}

The last step is the most interesting one: in order to compute the classifying object ${\mathcal BO} \, (k^n, \cdot) /\approx $ a new and generalized Jordan type theory must be developed. We called it \emph{$\cdot$-Jordan theory} and it has the following statement: let $(k^n, \cdot)$ be a given $4$-algebra. Two endomorphisms
$\Omega$ and $\Omega' \in {\rm End}_k (k^n)$ are called \emph{$\cdot$-similar}, and we denote this by $\Omega \approx^{\cdot} \, \Omega'$, if there exists $f \in {\rm Aut}_{\rm Alg} (k^n, \cdot)$ an automorphism of $(k^n, \cdot)$ and an element $v_0 \in k^n$ such that
$$
\Omega' (x) = (f \circ \Omega \circ f^{-1}) (x) - 2\, x \cdot v_0
$$
for all $x\in k^n$. We observe that for the trivial multiplication ($x \cdot y := 0$, for all $x$, $y\in k^n$) on $k^n$ the above equivalent relation is just the usual relation of similarity that appears in classical Jordan theory. We also mention that we left aside the last part of \equref{echividem} (namely $\Omega' (v_0) = v_0 - v_0^2$) since it is a normalizing type condition.

\textbf{Question 2:} \emph{For a given $4$-algebra structure $\cdot$ on $k^n$ describe explicitly a system of representatives for the above $\cdot$-similar relation.}

Of course, for the Bernstein problem in the last question it is enough to restrict to a system of representatives among all idempotent endomorphism of $k^n$. Finally, related to this step we ask the following question, the answer of which we intuitively expect to be affirmative:

\textbf{Question 3:} \emph{Let $\cdot$ be a given $4$-algebra structure on $k^n$. Is the set ${\mathcal BO} \, (k^n, \cdot) /\approx$ non-empty and does it have at most $n+1$ elements?}

\textbf{Acknowledgements.} The author thanks the referee for the helpful suggestions that improve the paper.


\begin{thebibliography}{99}

\bibitem{am-2019}
Agore, A.L. and Militaru, G. -  Extending structures. Fundamentals
and Applications, Taylor and Francis Group, Monographs and
Research Notes in Mathematics, 2019, 224 pages.

\bibitem{am-2016}
Agore, A.L. and Militaru, G. - Extending structures, Galois groups and supersolvable associative algebras, {\sl Monatsh. Math.}, {\bf 181} (2016), 1--33.

\bibitem{am-2015}
Agore, A.L. and Militaru, G. - The global extension problem, crossed products and co-flag non-commutative Poisson algebras,
{\sl J. Algebra}, {\bf 426} (2015), 1--31.

\bibitem{am-2013}
Agore, A.L., Militaru, G. - Classifying complements for Hopf algebras and Lie algebras, {\sl J. Algebra}, {\bf 391} (2013), 193--208.

\bibitem{bern1}
Bernstein, S.N. - Mathematical problems in modern biology, {\sl Science Ukraine}, {\bf 1} (1922),
14--19 (in Russian).

\bibitem{bern2}
Bernstein, S. N. - D\'{e}monstration math\'{e}matique de la loi d’h\'{e}r\'{e}dit\'{e} de Mendel, {\sl C. R.
Acad. Sci. Paris}, {\bf 177} (1923), 528--531.

\bibitem{bern3}
Bernstein, S. N. - Principle de stationarit\'{e} et g\'{e}n\'{e}ralisation de la loi de Mendel, {\sl C. R.
Acad. Sci. Paris}, {\bf 177} (1923), 581--584.

\bibitem{corte}
Cort\'{e}s, T. - Classification of $4$-dimensional Bernstein algebras, {\sl Comm. in Algebra}, {\bf 19} (1991), 1429--1443.

\bibitem{cortem}
Cort\'{e}s, T. and Montaner F. - Low dimensional Bernstein-Jordan algebras, {\sl J. London Math. Soc.}, {\bf 51} (1995) 53--61.

\bibitem{el}
Elduque, A. and Okubo, S. - On algebras satisfying $x^2x^2=N(x)x$, {\sl  Math. Z.}, {\bf 235} (2000), 275--314.


\bibitem{Eth}
Etherington, I. M. H. - Genetic algebras, {\sl Proceedings of the Royal Society of Edinburgh}, {\bf 59} (1939), 242--258.

\bibitem{gonz}
Gonz\'{a}lez, S., Guti\'{e}rrez, J.C. and Mart\'{\i}nez, C. - On regular Bernstein algebras, {\sl Linear
Algebra Appl.}, {\bf 241–243} (1996), 389--400.

\bibitem{gonz2}
Gonz\'{a}lez, S., L\'{o}pez-D\'{\i}az, M.C., Mart\'{\i}nez, C. and Shestakov, I.P. - Bernstein superalgebras and supermodules, {\sl J. Algebra}, {\bf 212} (1999), 119--131.

\bibitem{Fer1}
Guti\'{e}rrez-Fern\'{a}ndez, J.C. - Solution of the Bernstein problem in the non-regular case, {\sl J. Algebra}, {\bf 223} (2000) 109--132.

\bibitem{gonz}
Gonz\'{a}lez, S., Guti\'{e}rrez, J.C. and Mart\'{\i}nez, C. - On regular Bernstein algebras, {\sl Linear
Algebra Appl.}, {\bf 241–243} (1996), 389--400.

\bibitem{Fer2}
Guti\'{e}rrez-Fern\'{a}ndez, J.C. -  The Bernstein Problem in Dimension $6$, {\sl J. Algebra}, {\bf 185} (1996) 420--439.

\bibitem{guzzo}
Guzzo, H. Jr. and  Benh, A. - Solvability of a commutative algebra which atisfies $(x^2)^2=0$, {\sl Comm. in Algebra}, {\bf 42} (2014), 417--422.

\bibitem{Hol}
Holgate, P. - Genetic algebras satisfying Bernstein's stationarity principle, {\sl J. London Math. Soc.},  {\bf 9}(1975),  613--623.

\bibitem{Lj}
Lyubich, Yu. I. - Two-level Bernstein populations, {\sl Math. USSR Sb.}, {\bf 24}(1974), 593--615.

\bibitem{Lj1}
Ljubich, Yu. I. -  Algebraic methods in evolutionary genetics, {\sl Biom. J.}, {\bf 20}(1978), 511--529.

\bibitem{Lj01}
Lyubich, Yu. I. - A classification of some types of Bernstein algebras, {\sl Selecta
Mathematica Sovietica}, {\bf 6} (1987), l--14.

\bibitem{Lj0}
Lyubich, Yu. I. -  A New Advance in the Bernstein Problem in Mathematical Genetics, {\bf 9} (1996), Institute for Mathematical Science, Stony Brook,
arXiv:math/9608212.

\bibitem{mil22}
Militaru, G. - Crossed products of $4$-algebras. Applications, {\sl J. Algebra}, {\bf 614}(2023), 251--270.

\bibitem{PZ}
Piontkovski, D. and Zitan, F. - Bernstein algebras that are algebraic and the Kurosh problem, {\sl J. Algebra}, {\bf 617} (2023), 275--316.

\bibitem{reed}
Reed, M. L. (1997) - Algebraic structure of genetic inheritance, {\sl Bull. Amer. Math. Soc.}, {\bf 34} (1997), 107--130.

\bibitem{Sch}
Schafer, R. D. - Structure of genetic algebras, {\sl Amer. J. Math.}, {\bf 71} (1949),  121--135.

\bibitem{She}
Shestakov, I.P. and Zhukavets, N. - On associative algebras satisfying the identity $x^5=0$, {\sl Algebra Discrete Math.}, {\bf 1} (2004), 112--120.

\bibitem{Zh2}
Zhevlakov, K.A., Sliniko, A.M., Shestakov, I.P. and Shirsho, A.I. - Rings That Are Nearly Associative, Nauka, Moscow, 1978 (in
Russian); Academic Press, 1982 (English translation).

\bibitem{WB}
W\"{o}Rz-Busekros, A. -  Algebras in genetics, Lecture Notes in Biomathematics, Vol. {\bf 36}(1980), Springer-Verlag.

\bibitem{WB1}
W\"{o}Rz-Busekros, A. - Bernstein algebras, {\sl Arch. Math.}, Vol. {\bf 48}(1987), 388--398.

\bibitem{WB2}
W\"{o}Rz-Busekros, A. - Further remarks on Bernstein algebras, {\sl Proc. London Math.
Soc.}, {\bf 58} (1989), 69--73.

\end{thebibliography}
\end{document}